\documentclass{article}

\usepackage[T1]{fontenc}
\usepackage[utf8]{inputenc}
\usepackage[italian,english]{babel}

\DeclareUnicodeCharacter{2212}{-}

\usepackage{geometry}
\usepackage{emptypage}	
\usepackage{graphics}
\usepackage{graphicx}
\usepackage{amsmath}
\usepackage{amssymb}
\usepackage{amsthm}
\usepackage{graphicx}
\usepackage{wrapfig}
\usepackage{pdfpages}
\usepackage{braket}
\usepackage{enumerate}
\usepackage{hyperref}
\usepackage{mathtools}
\usepackage{subcaption}

\theoremstyle{plain}
\newtheorem{thm}{Theorem}[section]
\newtheorem{thm*}{Theorem}

\newtheorem{prop}[thm]{Proposition}

\theoremstyle{definition}
\newtheorem{dfn}[thm]{Definition}

\newtheorem{rmk}[thm]{Remark}

\newcommand{\R}{\mathbb{R}}

\newcommand{\s}{\mathcal{S}}
\newcommand{\abs}[1]{\lvert #1 \rvert}
\newcommand{\defeq}{\coloneqq}

\title{Relative equilibria, linear stability and electromagnetic curvature}
\author{
Luca Asselle and Giorgia Testolina\\[3pt]
\textit{Ruhr-Universität Bochum, Fakultät für Mathematik}
}

\date{\today}

\begin{document}

\maketitle

\vspace{-5mm}

\begin{abstract}
 \noindent In this paper we study the linear stability of relative equilibria in the Newtonian $n$-body problem from the viewpoint of electromagnetic systems. We first examine the effect of the ambient dimension on stability, starting from the Lagrange equilateral triangle solutions of the three-body problem in $\R^4$. We then initiate a new approach to stability based on electromagnetic curvature. In a two-dimensional model, we relate linear stability to both the Ma\~n\'e critical value and to the behavior of the zero set of the electromagnetic curvature, highlighting a change in its topology at the stability threshold. This criterion is then applied to the planar $n$-body problem: in the three-body case, we recover Routh's classical criterion, and, more generally, we obtain an instability criterion for relative equilibria whose reduced linearized dynamics splits along invariant symplectic planes. These results suggest a new geometric perspective on linear stability and on questions related to Moeckel's conjecture.
\end{abstract}

\section{Introduction}

Relative equilibria (RE) are among the most fundamental special solutions of the Newtonian $n$-body problem. They are motions for which the configuration rotates rigidly while preserving its shape, and (in dimension $\leq 3$) they arise from \textit{central configurations} (CC). 
Understanding how the dynamical stability of a RE is encoded in the variational properties of the CC generating it has been for decades one of the central problems in celestial mechanics. In the planar case, this question is embodied in two major conjectures: \textit{Moeckel's conjecture}, stating that CC generating linearly stable RE must be local minimizers of the Newtonian potential restricted to the inertia ellipsoid, and the \textit{dominant mass conjecture}, asserting roughly speaking that linearly stable RE can only arise if one of the masses is much larger than the others. 

Over the years, several important results have supported these conjectures. For planar relative equilibria, Routh's classical criterion completely describes the stability of the Lagrange equilateral triangle solution in the 3-body problem. More generally, Hu and Sun proved in \cite{Hu} that if the Morse index or the nullity of a planar CC is odd, then the corresponding RE is linearly unstable. This result was revisited and extended in \cite{Barutello} and later by the first author with Portaluri and Wu \cite{APW}, who obtained sufficient conditions for spectral instability of planar and non-planar RE and developed a symplectic decomposition of the linearized dynamics inspired by previous work of Meyer and Schmidt \cite{meyer-schimdt}. 
Regarding the dominant mass conjecture, Moeckel provided a criterion for linear stability when $n-1$ masses are small \cite{moeckel} and, more recently, Roberts in \cite{roberts} confirmed the validity of the conjecture for RE of the 4-body problem generated by kite CC. 
These advances, while significant, do not settle either conjecture, except in a few special cases. In particular, the spectral-flow approach of \cite{Hu,Barutello,APW} detects instability through parity information on Morse indices and nullities, and therefore appears to be intrinsically limited when one seeks to make progress toward Moeckel's conjecture.
In the present paper, we take first steps in two directions that, to the best of our knowledge, have not yet been explored in the literature.

First, we investigate the effect of the ambient dimension on the stability of RE. We begin with the simplest possible class of examples, namely the RE arising from Lagrange equilateral triangle CC for the 3-body problem in $\R^4$. Even in this elementary setting, varying the inclination of the generating CC produces a nontrivial stability diagram. More precisely, our numerical investigation (see Section 3) suggests that, for every choice of masses, the inclination angles for which the RE corresponding to a Lagrange equilateral triangle CC is linearly stable form an interval containing $[\pi/3,\pi/2]$. Moreover, this interval enlarges as the mass ratio approaches Routh's threshold, and the corresponding RE appears to be linearly stable for every inclination angle whenever Routh's criterion is satisfied. Although this part of the analysis is based on numerical computations rather than on a computer-assisted proof in the rigorous sense, it provides a clear picture of how strongly the ambient dimension can influence the spectrum of the linearized Hamiltonian system.  Here, we should also mention the related work of Albouy and Dullin \cite{albouydullin,albouy2dullin}. In particular, \cite{albouydullin} proves the existence of Lyapunov stable RE for the 3-body problem in $\mathbb R^4$.

Second, we initiate a new approach to the linear stability problem based on \textit{electromagnetic curvature}. After passing to rotating coordinates, RE become equilibrium points of an autonomous Hamiltonian system with an additional magnetic term encoding the Coriolis force. This places the stability problem within the framework of electromagnetic dynamics. In recent work \cite{assenza_testolina}, electromagnetic curvature was introduced as a natural geometric quantity for the study of the dynamical properties of electromagnetic systems in the low-energy regime. The main idea of the present paper is that this curvature can also encode relevant information about the linear stability of equilibrium points of electromagnetic systems, and hence, in particular, of RE in the $n$-body problem.
The core of our analysis is a two-dimensional model given by a linear electromagnetic system with quadratic potential. In Theorem~\ref{prop:equivalenceplane} we prove that, in this model, linear stability of the equilibrium is equivalent to a strict inequality between the Ma\~n\'e critical value and the maximal energy on the zero section, and also equivalent to a precise topological description of the zero set of the electromagnetic curvature near the equilibrium. This yields a geometric stability test that is both conceptually transparent and easy to evaluate.
We then apply this result to the planar $n$-body problem. In Proposition~\ref{prop:lagrange}, we show that, for the Lagrange equilateral triangle solution of the planar 3-body problem, the two-dimensional electromagnetic criterion exactly recovers Routh's classical stability condition. We then analyze the square and rhombus CC in the 4-body problem and, motivated by these examples, formulate in Proposition~\ref{prop:stabilitycriterion} a general instability criterion for those RE whose reduced linearized Hamiltonian splits into independent blocks on invariant symplectic planes. Although this splitting property is rather restrictive, the curvature point of view appears flexible enough to remain meaningful beyond the fully decoupled setting. This suggests a possible new strategy toward Moeckel-type conjectures: rather than extracting instability solely from parity information, one may try to infer stability or instability directly from geometric features of the reduced electromagnetic system (see the end of Section~5 for further details).

%The paper is organized as follows. In Section~2, we recall the basic material on CC, RE, and rotating coordinates. Section~3 is devoted to the stability of Lagrange RE in $\R^4$. In Section~4, we briefly introduce electromagnetic curvature and establish the two-dimensional criterion relating curvature and stability. Finally, in Section~5, we apply this criterion to planar RE.

\vspace{2mm}

\noindent \textbf{Acknowledgments:} The authors are partially supported by the DFG-grant 566804407 ``Symplectic dynamics, celestial mechanics and magnetism''.

\section{Preliminaries} \label{sec_2}

Consider $n$ point masses $m_1, \dots, m_n > 0$ with positions 
$q_1, \dots, q_n \in \R^d$, $d \geq 2$, moving under Newton's law of gravitation
\begin{equation}
\label{eqn_Newton}
\ddot{q} = M^{-1} \nabla U(q),
\end{equation}
where $q \defeq (q_1, \dots, q_n) \in \R^{dn}$ and 
$M \defeq \mathrm{diag}(m_1 \mathbb{I}_d, \dots, m_n \mathbb{I}_d)$ is the mass matrix. 
We restrict our attention to the collision-free configuration space with center of mass at the origin,
\[
\mathcal{C} = \{ q \in \R^{dn} \mid \bar{q} = 0,\; q_i \neq q_j \text{ for } i \neq j \},
\]
where 
$$\bar{q} := \frac{1}{\bar m}\sum_{i=1}^n m_i q_i, \quad \bar m := \sum_i m_i,$$ denotes the center of mass.

\begin{dfn} 
\label{dfn_CC}
A \textit{central configuration} (CC) is an arrangement $q$ of the masses such that
\begin{equation}
\label{eqn_CC_1}
M^{-1} \nabla U(q) + \lambda q = 0,
\end{equation}
where $\lambda > 0$ is a positive constant.
\end{dfn}

CC play a fundamental role in the study of the $n$-body problem for many reasons: for instance, they govern the behavior of colliding solutions near collisions, they mark changes in the topology of integral manifolds, and they give rise to explicit classes of periodic solutions. In low dimension, that is when $d \leq 3$, every CC generates a homothetic motion, in which the configuration collapses toward or expands from the center of mass while preserving its shape.  
Moreover, every \textit{planar} CC generates homographic motions, in which the configuration evolves only through scaling and rotation, remaining similar to its initial shape while each body follows a Keplerian orbit. As a special case, one retrieves \textit{relative equilibria} (RE), for which the configuration rigidly rotates about the center of mass with constant angular velocity and each body moves along a circular orbit.
In dimensions $d \geq 4$, the richer structure of the orthogonal group allows (possibly non-planar) CC to generate RE which are no longer confined to a single plane, see e.g. \cite{APW}. 

We introduce the mass scalar product $\langle \cdot, \cdot \rangle_M$ as
\[
\langle \cdot, \cdot \rangle_M \defeq \langle M \cdot, \cdot \rangle.
\]
Observing that $U(q)$ is $-1$-homogeneous and applying Euler's theorem for homogeneous functions, we are able to determine the value of the constant $\lambda$ in equation~\eqref{eqn_CC_1}:
\[
\lambda = \frac{U(q)}{\abs{q}^2_M}.
\]
From Definition~\ref{dfn_CC} we see that, if $q$ is a CC, then $k q$ is also a CC for any $k > 0$. Therefore, we can restrict our attention to the collision free configuration sphere
\[
\s \defeq \{ q \in \mathcal{C} \mid \abs{q}^2_M = 1 \}.
\]
The manifold $\s$ is an open subset of a smooth compact manifold diffeomorphic to a $d(n - 1)- 1$-dimensional sphere. Moreover, since in $\s$ we have $\lambda = U(q)$, Equation~\eqref{eqn_CC_1} becomes
\begin{equation}
\label{eqn_CC}
M^{-1} \nabla U(q) + U(q) q = 0.
\end{equation}
A fundamental feature of CC is their variational characterization as critical points of the Newtonian potential restricted to $\s$, denoted $U|_\s$. Since $U|_{\s}$ is invariant under the action of $SO(d)$, CC are never isolated critical points but instead occur in $SO(d)$-families. 
The Hessian of the restriction $U|_\s : \s \rightarrow \R$ at a critical point $q$ is the quadratic form on $T_q \s$ represented, with respect to the mass scalar product, by the $(dn \times dn)$-matrix
\[
H(q) = M^{-1} D^2 U(q) +  U(q), 
\]
where the $(i, j)$-blocks of $D^2 U(q)$ are given by 
\begin{align*}
D_{ij} & = \frac{m_i m_j}{r^3_{ij}} (\mathbb{I}_d - 3 u_{ij} u^t_{ij}) \quad \text{for} \, i \neq j \\
D_{ii} & = - \sum_{j \neq i} D_{ij}
\end{align*}
with $r_{ij} \defeq \abs{q_i - q_j} $ and $u_{ij} \defeq \frac{q_i - q_j}{\abs{q_i - q_j}}$. 

\subsection{Rotating frame for relative equilibria} \label{sec_2.1}

RE become equilibrium points when viewed in a uniformly rotating frame.
Their stability can therefore be investigated by studying the linearization of the equations of motion in rotating coordinates.
For this purpose it is convenient to adopt the Hamiltonian formulation of the $n$-body problem.
The equations of motion~\eqref{eqn_Newton} arise as the Euler-Lagrange equations associated with the Lagrangian $L: T\mathcal{C} \to \R$
\[
L(q,\dot q)=\tfrac12\lvert\dot q\rvert_M^2+U(q).
\]
Introducing the canonical momenta
\[
p_i \defeq \frac{\partial L}{\partial \dot q_i}=m_i\dot q_i,
\]
the Legendre transform yields the Hamiltonian
$H:T^*\mathcal{C}\to\R$
\[
H(q,p)=\tfrac12\lvert p\rvert_{M^{-1}}^2-U(q),
\]
and the dynamics is equivalently described by Hamilton's equations
\begin{equation}
\dot q=M^{-1}p, \qquad \dot p=\nabla U(q).
\end{equation}
Consider a uniformly rotating frame generated by
$R(t)=e^{tK}\in SO(d)$, where $K\in\mathfrak{so}(d)$ is a constant angular velocity matrix.
Introducing rotating coordinates
\[
Q=R(t)q, \qquad P=R(t)p,
\]
a direct computation shows that the Hamiltonian becomes
\begin{equation}\label{eqn_ham_rot}
H(Q,P)
=\tfrac12\lvert P\rvert_{M^{-1}}^2
+\langle KQ, P\rangle
-U(Q),
\end{equation}
and Hamilton's equations take the form
\begin{equation}\label{eqn_hamilton}
\dot Q=M^{-1}P+K Q,
\qquad
\dot P=\nabla U(Q)+K P.
\end{equation}
The additional terms encode the Coriolis and centrifugal effects
induced by the rotating frame and yields a Hamiltonian of electromagnetic type (see Section~\ref{sec_4} for more details). Indeed,~\eqref{eqn_ham_rot} can be rewritten as 
\[
H(Q, P) = \tfrac12 \lvert P - \vartheta_Q \rvert_{M^{-1}}^2 + V(Q)
\]
where
\[
\vartheta_Q \defeq -MKQ
\]
plays the role of a magnetic potential, and
\[
V(Q) \defeq - U(Q) + \tfrac12  \lvert KQ \rvert_M^2
\]
is the corresponding effective potential. RE of the form
\[
q(t)=R(t) q,
\]
with $q$ a CC, correspond to equilibrium points of the rotating system precisely when the angular velocity matrix $K$ satisfies the condition
\[
-K^2 q =\lambda q,
\]
where $\lambda$ is the constant appearing in the CC equation~\eqref{eqn_CC}. 
Linearizing system~\eqref{eqn_hamilton} at an equilibrium $z = (Q,P)$ yields the linear Hamiltonian system
\begin{equation}\label{linear_system}
\dot z = L z = -J B z,  
\end{equation}
where
\[
L=
\begin{pmatrix}
K & M^{-1}\\
D^2U(Q) & K
\end{pmatrix},
\qquad
B=
\begin{pmatrix}
- D^2U(Q) & K^T\\
K & M^{-1}
\end{pmatrix},
\]
and $J$ denotes the standard symplectic matrix on $\R^{2dn}$.
Because of the symmetries of the problem, the phase space admits a decomposition into invariant
symplectic subspaces 
$E_1 \oplus E_2 \oplus E_3$,
where $E_1$ and $E_2$ correspond to the translational and rotational symmetries, while $E_3$ contains the dynamically
nontrivial directions. We therefore define a RE to be \emph{linearly stable} if the spectrum of the linearized system,
restricted to $E_3$, is purely imaginary and the matrix is diagonalizable.

\section{Linear stability in $\R^4$} \label{sec_3}

In this section, we investigate the effect of the ambient dimension on the stability of RE, beginning with the simplest possible class of examples, namely the RE arising from Lagrange equilateral triangle CC for the 3-body problem in $\R^4$. We first recall briefly the properties of the rotation group in $\R^4$ and the symplectic decomposition of the phase space for RE, referring to~\cite{APW} for a more comprehensive explanation. 

\subsection{Rotating frame for relative equilibria in $\R^4$} \label{sec_3.1}

The orthogonal group in four dimensions exhibits a much richer structure than in dimension two and three, allowing for instance rotations in two orthogonal planes with possibly different angular velocities. 
Any rotation $R \in \mathrm{SO}(4)$ fixes the origin and, after a suitable choice of orthonormal coordinates, can be written in block--diagonal form as
\[
R = \begin{pmatrix}
R(\varphi_1) & 0 \\
0 & R(\varphi_2)
\end{pmatrix},
\]
where each block $R(\varphi)$ is a $2\times 2$ planar rotation matrix
\[
R(\varphi) = 
\begin{pmatrix}
\cos \varphi & -\sin \varphi \\
\sin \varphi & \cos \varphi
\end{pmatrix} = e^{\varphi J},
\qquad
J=
\begin{pmatrix}
0 & -1\\
1 & 0
\end{pmatrix}.
\]
We distinguish the following types of rotations:
\begin{itemize}
    \item A \emph{simple rotation}, corresponding to the case in which one of the angles vanishes. The motion is then confined to a single invariant plane, while the orthogonal plane remains fixed.
    
    \item A \emph{double rotation}, occurring when both angles are nonzero. In this case the origin is the only fixed point, and the two orthogonal planes are invariant.
    A special subclass of double rotations is given by the \emph{isoclinic rotations}, for which the two rotation angles have equal absolute value. These are called \emph{left--isoclinic} when $\varphi_1 = -\varphi_2$ and \emph{right--isoclinic} when $\varphi_1 = \varphi_2$. In either case, there exist infinitely many invariant planes.
\end{itemize}

\noindent
As a consequence, the choice of uniformly rotating coordinates used to describe a relative equilibrium depends on the geometry of the underlying configuration. 
In particular, if all the masses lie entirely in one of the principal rotational planes -- say in the plane $\R^2 \times \{0\} $ -- we consider the simple rotation with angular velocity $k\defeq\sqrt{\lambda}$, where $\lambda$ is given by Equation~\eqref{eqn_CC}:
\[
r_s(t)=
\begin{pmatrix}
e^{ktJ} & 0\\
0 & I_2
\end{pmatrix}
= \text{exp}(k_s t),
\qquad
k_s=
\begin{pmatrix}
kJ & 0\\
0 & 0
\end{pmatrix}\in\mathfrak{so}(4).
\]
We denote by $R_s(t)$ and $K_s$ the $4n\times4n$ block-diagonal matrices whose diagonal entries are respectively the $4\times4$ matrices $r_s(t)$ and $k_s$.
If instead the masses are not contained in one of the principal rotation planes we define the left-isoclinic rotation
\[
r_d(t)=
\begin{pmatrix}
e^{ktJ} & 0\\
0 & e^{-ktJ}
\end{pmatrix}
= \text{exp}(k_d t),
\qquad
k_d=
\begin{pmatrix}
kJ & 0\\
0 & -kJ
\end{pmatrix}\in\mathfrak{so}(4),
\]
and denote by $R_d(t)$ and $K_d$ the corresponding
$4n\times4n$ matrices. 

\subsection{Symplectic decomposition of the phase space} \label{sec_3.2}

The symmetries of the $n$-body problem induce a symplectic decomposition of the phase space $\R^{8n}$ of the linearized system~\eqref{linear_system} into three invariant symplectic subspaces
\[
E_1 \oplus E_2 \oplus E_3.
\]
The subspace $E_1$ corresponds to translational invariance and is generated by motions of the center of mass together with their conjugate momenta.
The subspace $E_2$ is associated with rotational symmetry. It is spanned by infinitesimal rotations of the configuration together with their conjugate momenta. Its structure depends on the type of central configuration and on whether the corresponding relative equilibrium is planar or non-planar. More precisely, if the central configuration is planar and non-collinear and the corresponding relative equilibrium is planar
(i.e. the configuration is contained in a plane and the motion happens inside that plane), $E_2$ has dimension 12 and is generated by the six independent simple rotations of $\R^4$, together with their conjugate momenta. If instead the motion is non-planar, or if the central configuration is collinear, the rotational symmetry reduces to simultaneous rotations in orthogonal planes. In this case $E_2$ has dimension 8 and is generated by the simultaneous rotations in two orthogonal planes and their conjugate momenta.
Finally, $E_3$ is defined as the symplectic complement of $E_1\oplus E_2$ and contains all the dynamically relevant information.

Starting from this decomposition, it is possible to construct a symplectic change of coordinates which allows one to split the original Hamiltonian system into three Hamiltonian subsystems, each defined on the corresponding symplectic subspace of the above decomposition. Let $A$ be an $M$-orthogonal $4n \times 4n$ matrix commuting with $K$. Then the linear transformation 
\[
(q, p) \mapsto (A^{-1}q, A^{-\top}p)
\]
is symplectic and transforms the linearized system into a block form adapted to the decomposition $E_1 \oplus E_2 \oplus E_3$. The matrix $A$ is obtained by choosing an $M$-orthonormal basis of $\R^{4n}$, adapted to this decomposition: the first vectors generate the translational and rotational subspaces, while the remaining ones complete them to an $M$-orthonormal basis of $\R^{4n}$. 

\subsection{Stability of Lagrange equilateral triangle RE} \label{sec_3.3}

We apply the framework discussed in the previous sections to analyze the linear stability of the relative equilibria generated by Lagrange equilateral triangle central configurations.

Thus, let $q \in \s$ be a Lagrange equilateral triangle CC lying in the plane 
$\R^2 \times \{0\}$, and let $q(\gamma)=R(\gamma)q$ denote the CC
obtained by rotating $q$ by an angle $\gamma$ in the $xz$–plane,
where $R(\gamma)$ is the corresponding rotation matrix. The case $\gamma=0$ corresponds to the initial configuration lying in the plane $xy$, while $\gamma=\tfrac{\pi}{2}$ corresponds to a configuration contained in the $xz$-plane. By rotational invariance and reflection symmetry, it suffices to restrict to $\gamma \in [0,\pi/2]$. 
 We define the \textit{stability region} of $q$ as 
$$I(q) := \big \{\gamma \in [0,\pi/2] \ \big |\ \text{the RE corresponding to } q(\gamma) \text{ is linearly stable}\big \}.$$
Our numerical experiments suggest the following behavior: for every configuration $q \in S$ that we tested, the stability region $I(q)$ appears to be an interval containing $[\pi/3,\pi/2]$. Moreover, this interval appears to enlarge as the mass ratio approaches Routh's stability criterion
\[
\frac{m_1m_2 + m_1m_3 + m_2m_3}{(m_1 + m_2 + m_3)^2}
<
\frac{1}{27},
\]
and the computations indicate that $I(q) = [0,\pi/2]$ whenever Routh's criterion is satisfied.

The discussion above is based on numerical computations of the spectrum of the linearized Hamiltonian system restricted to the invariant subspace $E_3$. Since these computations are not computer-assisted in the rigorous sense, we present them as numerical evidence rather than as a formal proof. %We emphasize that the main results of Sections~4 and~5 do not rely on this numerical observation.
More precisely, we place three masses at the vertices of an equilateral triangle in the $xy$-plane with coordinates
\[
q_1 = (1,0,0,0), \qquad
q_2 = \Bigl(-\frac12,-\frac{\sqrt3}{2},0,0\Bigr), \qquad
q_3 = \Bigl(-\frac12,\frac{\sqrt3}{2},0,0\Bigr).
\]
The configuration is then translated so that the center of mass lies at the origin and normalized so as to satisfy $q^\top M q = 1$. The rotation matrix $R(\gamma)$ defining $q(\gamma)$ is given by
\[
R(\gamma)=
\begin{pmatrix}
\cos\gamma & 0 & \sin\gamma & 0\\
0 & 1 & 0 & 0\\
-\sin\gamma & 0 & \cos\gamma & 0\\
0 & 0 & 0 & 1
\end{pmatrix},
\qquad \gamma \in [0,\pi/2].
\]
We construct the linearized system~\eqref{linear_system} and the symplectic change of coordinates described above, restricting the system to the invariant subspace $E_3$. The spectrum of the resulting linear system is then computed while varying $\gamma$ over a discrete grid in $[0,\pi/2]$, tracking the real and imaginary parts of the eigenvalues at each step. The computation is repeated for several choices of mass distributions, leading to the stability behavior described above. %The plots in Figure~1 correspond to the cases
%\[
%\text{(a) } m_1=m_2=m_3=1, \qquad
%\text{(b) } m_1=30,\; m_2=m_3=1, \qquad
%\text{(c) } m_1=50.46,\; m_2=m_3=1.
%\]
Figure~1 displays the maximal real part of the spectrum of the linearized system as a function of the inclination angle. 
\begin{figure}[h!]
    \centering

    \begin{subfigure}[b]{0.48\textwidth}
        \centering
        \includegraphics[width=\linewidth]{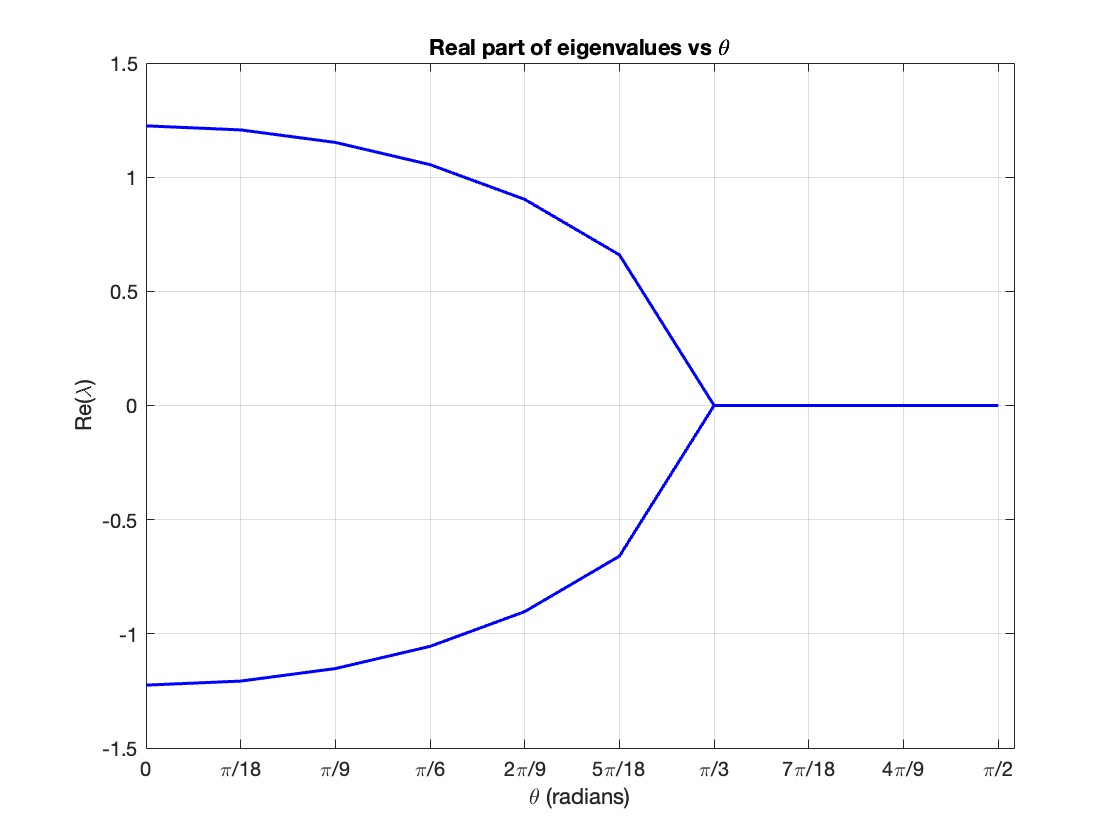}
        \caption{$m_1=m_2=m_3=1$}
        \label{fig:eig_equal}
    \end{subfigure}
    \hfill
    \begin{subfigure}[b]{0.48\textwidth}
        \centering
        \includegraphics[width=\linewidth]{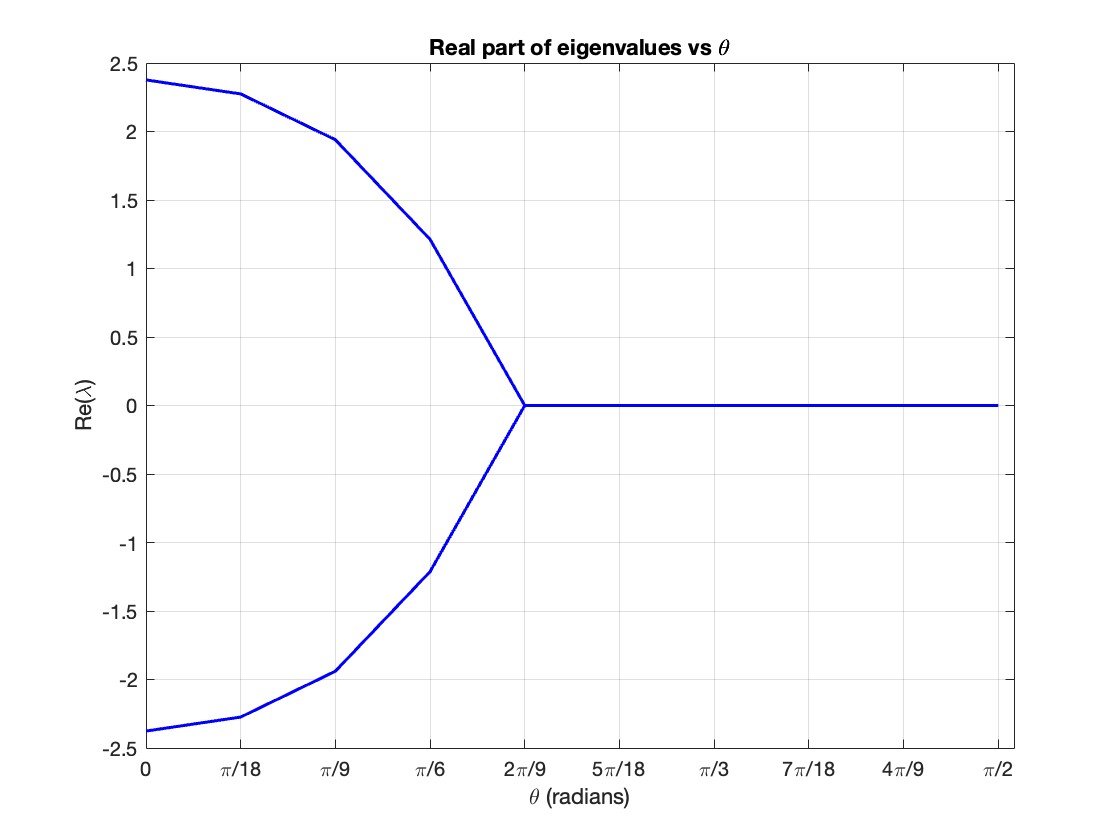}
        \caption{$m_1=30,\; m_2=m_3=1$}
        \label{fig:eig_m1_30}
    \end{subfigure}

    \vspace{0.5cm}

    \begin{subfigure}[b]{0.48\textwidth}
        \centering
        \includegraphics[width=\linewidth]{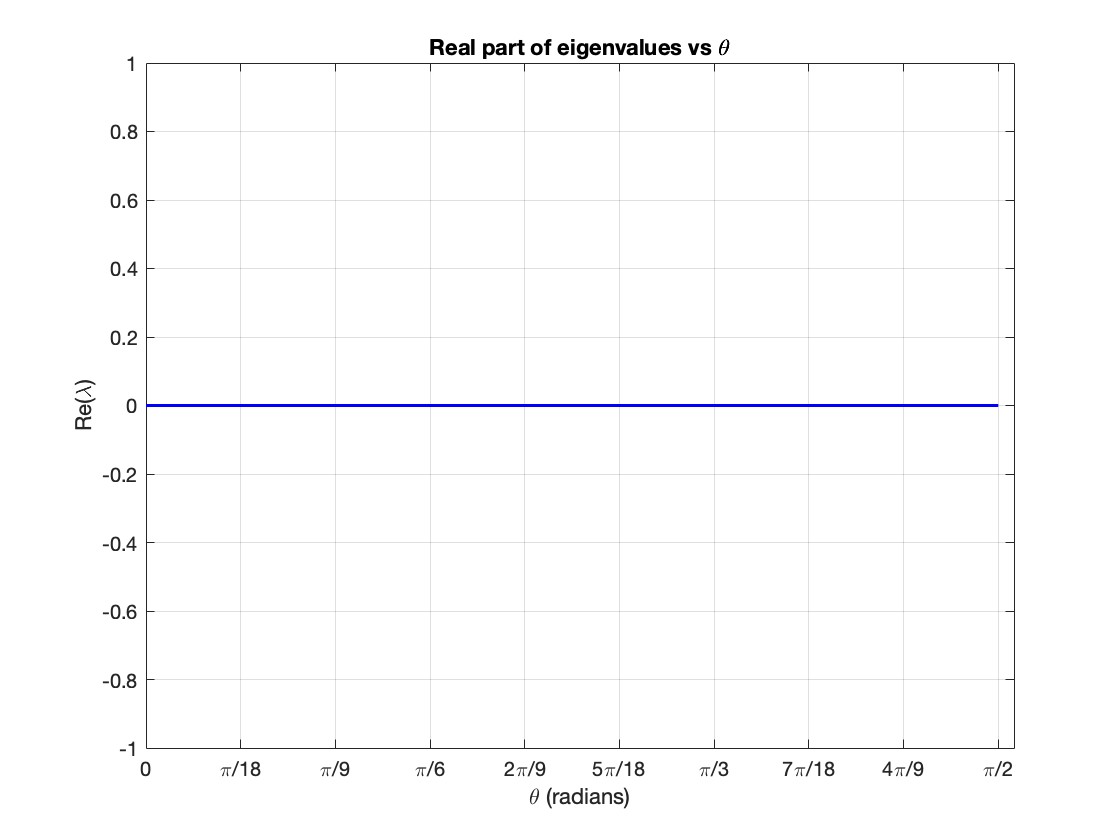}
        \caption{$m_1=50.46,\; m_2=m_3=1$.}
        \label{fig:eig_m1_50}
    \end{subfigure}

    \caption{Real part of the eigenvalues for different mass ratios. Fixing $m_2 = m_3 = 1$, Routh’s criterion holds for $m_1>50.456$.}
    \label{fig:eig_comparison}
\end{figure}

\noindent In Figure~1a, corresponding to the case of three equal masses $m_1=m_2=m_3=1$, the maximal real part vanishes precisely at $\gamma=\pi/3$ and remains zero for all $\gamma \in [\pi/3,\pi/2]$. 
Keeping $m_2=m_3=1$, the other two plots show the effect of increasing the mass $m_1$: in Figure~1b, $m_1$ is set to $30$ and the stability region is visibly enlarged, while in Figure~1c, $m_1$ is increased to $50.46$, a value close to the threshold predicted by Routh's criterion $m_1 > 25 + 18\sqrt2 \approx 50.456,$
for which the configuration appears to be stable for all inclination angles. In all cases, in the whole stability region, apart from the two zero eigenvalues discussed in Remark~3.2 below, the remaining eigenvalues are numerically purely imaginary and simple, which is consistent with linear stability. 

\begin{rmk}
    Unlike in Routh's criterion for planar Lagrange RE, we do not observe a transition from instability to linear stability through spectral stability here.
    \end{rmk}

\begin{rmk}
The spectrum of the linearized Hamiltonian system restricted to $E_3$ always contains two zero eigenvalues, due to the part of the rotational symmetry that we are not able to eliminate even after the symplectic decomposition, see~\cite[Section 2.3]{APW}.
\end{rmk}

\begin{rmk} CC in $\R^4$ do not exhaust the class of configurations generating relative equilibria. Because rotations can occur simultaneously in two orthogonal planes with different angular velocities, the  broader class of \textit{balanced configurations}  generates relative equilibria. These were introduced by Albouy and Chenciner in \cite{Albouy_Chen} (see also~\cite{Moeckel2014,AP}). In particular, isosceles triangles with equal masses at the symmetric vertices give rise to RE for the $3$-body problem in $\R^4$.  The study of their stability features can be found in~\cite{Giorgia}.
\end{rmk}

\section{Electromagnetic curvature and stability} \label{sec_4}

Let $(M, g)$ be an oriented Riemannian manifold. An (exact) electromagnetic system on $(M, g)$ consists of a magnetic potential given by a smooth 1-form $\vartheta \in \Omega^1(M)$, and a scalar potential $V \in C^\infty(M)$. The electromagnetic Lagrangian is given by
\[
L(q, v) = \tfrac12 |v|_g^2 + \vartheta_q[v] - V(q), \qquad (q, v) \in TM,
\]
and defines the electromagnetic flow on $TM$ via the Euler-Lagrange equations. The electromagnetic flow corresponds via the Legendre transform to the Hamiltonian flow defined by 
\[
H(q, p) = \tfrac12 \lvert p - \vartheta_q \rvert_g^2 + V(q), \qquad (q, p) \in T^*M
\]
where, with a slight abuse of notation, we denote by $\lvert \cdot \rvert_g$ also the dual norm.
The associated magnetic field is the (exact) 2-form $\sigma = d\theta$. 
The energy function 
\[
E(q, v) = \frac12 |v|_g^2 + V(q)
\]
is conserved along the electromagnetic flow and takes values in the interval $[\inf_M V, + \infty)$. Among the energy levels in $[\inf_M V, + \infty)$, two values mark significant changes in the dynamical properties of the electromagnetic flow, namely the maximum of the energy over the zero section
\[
e_0 \defeq \sup_{\substack{q \in M}} V(q).
\]
and the Mañé critical value of the universal cover
\[
c_u \defeq \inf_{f \in C^\infty(\tilde M)} \sup_{\tilde q \in \tilde M} \tilde H(\tilde q, \mathrm{d}f(\tilde q)),
\]
where $\tilde H: T^*\tilde M \rightarrow \R $ is the lift of the Hamiltonian $H$ to the universal cover $\tilde M$ of $M$. For energy values $e \geq e_0$, the footpoint projection $\pi:E^{-1}(e) \to M$ is surjective, whereas for energies below $e_0$ it is not, and the dynamics takes place in the so-called Hill's region. The value $c_u$, instead, characterizes the energy threshold above which the free-period action functional is bounded from below and satisfies the Palais--Smale condition; see, e.g., \cite{abbo2013}. It is well known that
\[
\inf_M V \leq e_0 \leq c_u,
\]
with the first inequality being an equality if and only if $V \equiv 0$. It is a more delicate problem to characterize the cases in which $e_0 = c_u$. Under the assumption that $V \equiv 0$ (that is, the system is purely magnetic), one can show that $e_0 = c_u$ if and only if $\sigma \equiv 0$ (that is, $\vartheta$ is closed). We shall see below that, for $V \not\equiv 0$, there is a connection between the condition $e_0 < c_u$ and the stability properties of equilibrium points. For more general conditions ensuring the strict inequality $e_0<c_u$ we refer to \cite{ABM}.

A fundamental problem in the study of electromagnetic systems is the existence of periodic orbits on a prescribed energy level. Above $c_u$, classical variational methods guarantee the existence of periodic trajectories on every energy level. In contrast, for energies below $c_u$, periodic orbits are known to exist only for almost every energy level; see \cite{abbo2013,lus-fet}. Recently, in \cite{assenza_testolina}, a notion of
\emph{electromagnetic curvature} was introduced to further investigate the dynamics in the energy range $[e_0,c_u]$. In particular, it is shown that, for energies sufficiently close to $e_0$, the existence of periodic orbits is guaranteed on every energy level as soon as the potential is $C^2$-small.

We now briefly recall the definition of electromagnetic curvature, referring to \cite{assenza_testolina} for further details. Using the Jacobi--Maupertuis principle, electromagnetic trajectories with energy $e > e_0$ are equivalent, up to a time reparametrization, to magnetic geodesics of the conformal metric
\[
g_e \defeq 2(e - V)g.
\]
The electromagnetic curvature is then defined as the magnetic curvature (see \cite{assenza}) of the purely magnetic system $(g_e, \sigma)$, pulled back to the original energy level $\Sigma_e \defeq E^{-1}(e)$. In the particular case in which $M$ is a surface, the electromagnetic curvature at energy $e$ reads
\begin{equation}\label{curvature}
\begin{split}
K_e^{g,b,V}(q,v)
&= \frac{K^g(q)}{2(e-V(q))}
+ \frac{\Delta_g V(q)}{4(e-V(q))^2}
+ \frac{|\mathrm{d}V(q)|_g^2}{4(e-V(q))^3}
- \frac{\mathrm{d}b(q)[\mathrm{J}^g \hat v]}{(2(e-V(q)))^{\frac32}} \\
&\quad
- \frac{2b(q)\,\mathrm{d}V(q)[\mathrm{J}^g \hat v]}{(2(e-V(q)))^{\frac52}}
+ \frac{b^2(q)}{4(e-V(q))^2},
\end{split}
\end{equation}
where $(q, v) \in \Sigma_e$, $\hat{v} \defeq v / \lvert v \rvert_g$, $K^g$ is the Gaussian curvature of $(M, g)$, $\Delta_g V$ is the Laplacian of $V$, $J^g$ is the complex structure induced by $g$, and $\sigma = b\mu_g$ with magnetic function $b \in C^\infty(M)$.

\begin{rmk}
The definition requires $e > e_0$ only in order for the Jacobi--Maupertuis metric $g_e$ to be globally well defined. For linear stability, however, this restriction is inessential. Indeed, when studying the linear stability of an equilibrium point, one linearizes the dynamics near that point; in particular, the electromagnetic curvature is only needed on a neighborhood contained in the open set $\{V<e\}$, and hence the degeneracy of the Jacobi--Maupertuis metric at the boundary of the Hill's region plays no role. As a consequence, the same local construction, and in particular the arguments behind Theorem \ref{prop:equivalenceplane}, and the possible higher-dimensional generalizations discussed after the proof, can also be used for equilibrium points that are not local maxima of the potential. 
\end{rmk}

\subsection{A two-dimensional model} \label{sec_4.1}

In this section, we study a simple (linear) electromagnetic system on the plane that serves as a reference for what follows. In particular, we analyze the relationship between the linear stability of an equilibrium point and electromagnetic curvature. Information about nonlinear systems is then obtained, as usual, by linearizing the dynamics at a given equilibrium point.
Thus, consider the plane $\mathbb{R}^2$ endowed with the standard Euclidean metric, and let $q=(x,y)$ denote the configuration variable. We study the electromagnetic system defined by the quadratic potential
\[
V(x,y) = -\alpha x^2 - \beta y^2, \qquad \alpha,\beta > 0,
\]
and the magnetic potential
\[
\vartheta = -y\,dx + x\,dy.
\]
The magnetic field is given by
\[
\sigma = d\vartheta = 2\,dx \wedge dy,
\]
hence the magnetic function is constant, $b \equiv 2$. The corresponding Hamiltonian is
\[
H(q,p) = \tfrac{1}{2}\abs{p - \vartheta_q}^2 + V(q)
= \tfrac{1}{2}\big[(p_x + y)^2 + (p_y - x)^2\big] - \alpha x^2 - \beta y^2,
\]
where $p=(p_x,p_y)$ denotes the momentum coordinates. The associated energy function is
\[
E(q, \dot q) = \tfrac{1}{2} \lvert \dot q \rvert^2 + V(q).
\]
Observe that the potential $V$ has a non-degenerate maximum at $q_0=(0,0)$, with $V(q_0)=0$. In particular, $e_0=0$. The following proposition describes the relationship between the linear stability of the origin $q_0$ as an equilibrium point of the electromagnetic system, the Ma\~n\'e critical value $c_u$, and the behavior of the electromagnetic curvature in a neighborhood of $q_0$. To ease the notation, hereafter we write $K_e$ instead of $K_e^{g,b,V}$.

\begin{thm}
The following assertions are equivalent:
\begin{enumerate}
\item $q_0$ is a linearly stable equilibrium point.
\item $c_u>e_0$.
\item For every $e>0$, we have $K_e(q_0,\cdot)>0$ and, denoting by $\mathcal C_e$ the connected component of  $\{q \in \R^2 \mid K_e(q,\cdot)>0\}$ containing $q_0$, we have that $\partial \mathcal C_e$
is either empty or given by the two branches of a hyperbola.
\item $\sqrt{\alpha} + \sqrt{\beta} < \sqrt{2}$.
\end{enumerate}
\label{prop:equivalenceplane}
\end{thm}

Before proving the theorem, we make some comments. First, in the purely mechanical case (i.e.\ when $\sigma = 0$), a non-degenerate maximum of the potential is always hyperbolic. The presence of a magnetic field can change the nature of the equilibrium point, making it linearly stable. The theorem above quantifies, in dimension two, how strong the magnetic field must be in order for the equilibrium point to become linearly stable. Second, the change in the stability properties of the equilibrium point occurs precisely when the topology of $\partial \mathcal C_e$ changes. Third, the condition $\alpha + \beta < 2$ is necessary for the electromagnetic curvature to be positive in a neighborhood of the origin. When $\alpha + \beta < 1$, the magnetic field dominates the potential, the equilibrium is linearly stable, and the electromagnetic curvature is positive everywhere, so that $\partial \mathcal C_e = \emptyset$. In the range $1 \le \alpha+\beta <2$, the strengths of the magnetic field and the potential are comparable, and positivity of the electromagnetic curvature is no longer guaranteed everywhere.

\begin{proof}
We first check the linear stability condition. The system is linear, and Hamilton's equations are given by the matrix
\[
L=\begin{pmatrix}
0 & 0 & 1 & 0\\
0 & 0 & 0 & 1 \\
2\alpha & 0 & 0 & 2\\
0& 2\beta & -2 & 0
\end{pmatrix}.
\]
The characteristic polynomial of $L$ is
\[
\lambda^4 + (4 - 2(\alpha + \beta)) \lambda^2 + 4\alpha\beta = 0,
\]
and the equilibrium is linearly stable if and only if all zeros are purely imaginary and $L$ is diagonalizable.
Setting $\mu = \lambda^2$, we see that this is equivalent to requiring that
\[
\mu_{\pm} = \alpha + \beta - 2 \pm \sqrt{(2 - \alpha - \beta)^2 - 4\alpha\beta} <0
\]
are both negative and distinct\footnote{One can easily check that, if $\mu_{+}=\mu_- <0$, which happens if and only if $\sqrt{\alpha}+\sqrt{\beta}=\sqrt{2}$, then the equilibrium is spectrally stable but not linearly stable.}, which happens if and only if
$$\alpha+\beta-2 <0, \quad \text{and}\quad (2-\alpha-\beta)^2 - 4\alpha\beta >0.$$
It is now straightforward to check that this is equivalent to
\[
\sqrt{\alpha} + \sqrt{\beta} < \sqrt{2}.
\]

We now determine for which values of $\alpha$ and $\beta$ the strict inequality $c_u>e_0$ holds. Recall that $c_u$ is defined by
$$c_u = \inf_{f\in C^\infty(\R^2)} \sup_{q\in \R^2} H(q,\mathrm d f(q)).$$
Taking $f_C:\R^2\to \R$, $f_C(x,y) = Cxy$, for some $C\in \R$, we obtain
\[
H(q, \mathrm d f_C(q)) = \tfrac{1}{2}(C - 1)^2 x^2 + \tfrac{1}{2}(C + 1)^2 y^2 - \alpha x^2 - \beta y^2.
\]
We readily see that there exists $C_*\in \R$ such that $H(\cdot,\mathrm d f_{C_*}(\cdot)) \leq 0$ if and only if the system
$$\left \{\begin{array}{r}  (C - 1)^2 \leq 2\alpha, \\ (C + 1)^2 \leq 2\beta, \end{array}\right.$$
has a solution. This happens if and only if $\sqrt{\alpha} + \sqrt{\beta} \geq \sqrt{2}$.
Therefore, if $\sqrt{\alpha}+\sqrt{\beta} \geq \sqrt{2}$, then
\begin{align*}
0\leq c_u &= \inf_{f\in C^\infty(\R^2)} \sup_{q\in \R^2} H(q,\mathrm d f(q))\\
        &\leq \inf_{C\in \R} \sup_{q\in \R^2} H(q,\mathrm d f_C(q))\\
        &\leq \sup_{q\in \R^2} H(q,\mathrm d f_{C_*}(q))\\
        &=0,
\end{align*}
thus $c_u=0$. To prove the reverse implication, we assume that $\sqrt{\alpha}+ \sqrt{\beta}< \sqrt{2}$ and show that in this case $c_u>0$. For this purpose, it is convenient to use the equivalent definition of $c_u$ in terms of the free-period Lagrangian action functional (see, e.g., \cite{abbo2013})
$$\mathcal S_\kappa (\gamma) := \int_0^{T(\gamma)} \Big [ L(\gamma(t),\dot \gamma(t)) + \kappa \Big ]\mathrm d t, \quad \gamma : \R/T(\gamma)\mathbb Z \to \R^2, \ \kappa \in \R,$$
namely
$$c_u := \sup \Big \{ \kappa \in \R \ \Big |\ \mathcal S_\kappa (\gamma) <0, \ \text{for some loop} \ \gamma\Big \}.$$
Therefore, all we have to show is that there exists a loop $\gamma$ such that $\mathcal S_0(\gamma)<0$. We choose
$$\gamma_{\alpha,\beta}(t) := \Big (\beta^{1/4} \cos \omega t , - \alpha^{1/4} \sin \omega t \Big), \quad \omega:= \sqrt{2} (\alpha\beta)^{1/4},$$
which gives a clockwise parametrization of the ellipse\footnote{Since the system is linear, scaling the ellipse by a factor $\rho$ only multiplies the action $\mathcal S_0$ by a factor $\rho^2$. In particular, it does not change the sign of $\mathcal S_0$.} centered at the origin, with semiaxes $\beta^{1/4}$ and $\alpha^{1/4}$, respectively. The period of $\gamma_{\alpha,\beta}$ is given by $T=\frac{2\pi}{\omega} = \frac{\pi\sqrt{2}}{(\alpha\beta)^{1/4}}$. A straightforward computation shows that
$$\mathcal S_0(\gamma_{\alpha,\beta}) = 2\sqrt{2}\pi (\alpha\beta)^{1/4} \Big ( \sqrt{\alpha}+\sqrt{\beta} - \sqrt{2}\Big ) <0,$$
as claimed. We note that a similar computation with $\gamma_{\alpha,\beta}$ replaced by a suitable parametrization $\gamma_{\text{cir}}$ of the circle of radius $1$ centered at the origin yields $\mathcal S_0(\gamma_{\text{cir}})<0$ if and only if $\alpha+\beta<1$, which is precisely the condition implying that the electromagnetic curvature is positive everywhere (and hence $\partial \mathcal C_e=\emptyset$), as we will see below.

We now study the electromagnetic curvature for energies $e > 0$ in detail.  For a unit tangent vector $v = (\cos\varphi, \sin\varphi)$, \eqref{curvature} gives
\[
K_e(q,v) =
\frac{2 - \alpha - \beta}{2 (e + \alpha x^2 + \beta y^2)^2}
+ \frac{\alpha^2 x^2 + \beta^2 y^2}{(e + \alpha x^2 + \beta y^2)^3}
- \frac{\sqrt{2}\,\big(\alpha x \sin\varphi - \beta y \cos\varphi\big)}{(e + \alpha x^2 + \beta y^2)^{5/2}}.
\]
At the equilibrium point $q_0$, this simplifies to
\[
K_e(q_0,v) = \frac{2 - (\alpha + \beta)}{2e^2},
\]
which is positive if and only if $\alpha + \beta < 2$. Therefore, hereafter we assume that $\alpha+\beta<2$. We are now interested in understanding how the behavior of the electromagnetic curvature in a neighborhood of $q_0$ for energies $e \to 0^+$ is affected by $\alpha$ and $\beta$.
To estimate the size of the region where $K_e > 0$, we first minimize the curvature over $\varphi$. This yields
\[
K_e(q,v) \ge
\frac{2 - \alpha - \beta}{2 (e + \alpha x^2 + \beta y^2)^2}
+ \frac{\alpha^2 x^2 + \beta^2 y^2}{(e + \alpha x^2 + \beta y^2)^3}
- \frac{\sqrt{2}\, \sqrt{\alpha^2 x^2 + \beta^2 y^2}}{(e + \alpha x^2 + \beta y^2)^{5/2}} =: \widetilde{K_e}(q).
\]
Introducing the parameter
\[
t = \sqrt{\frac{\alpha^2 x^2 + \beta^2 y^2}{e + \alpha x^2 + \beta y^2}},
\]
we can write
\[
\widetilde {K_e}(q) = \frac{1}{(e + \alpha x^2 + \beta y^2)^2}
\Bigl( t^2 - \sqrt{2}\,t + \tfrac{2 - (\alpha + \beta)}{2} \Bigr).
\]
Thus, the sign of $\widetilde{K_e}$ is governed by the quadratic polynomial
\[
Q(t) = t^2 - \sqrt{2}\,t + \tfrac{2 - (\alpha + \beta)}{2}.
\]
The minimum of $Q$ occurs at $t_* = \frac {1}{\sqrt{2}}$ and satisfies
\[
Q(t_*) = \tfrac{1 - (\alpha + \beta)}{2},
\]
which is positive as soon as $\alpha + \beta < 1$. In this case, $\widetilde{K_e}> 0$ everywhere; hence, in particular, the electromagnetic curvature $K_e$ is always positive, as mentioned above.
If $1 \le \alpha + \beta < 2$, the polynomial $Q$ has two positive roots
\[
t_{\pm} = \frac{1 \pm \sqrt{\alpha + \beta - 1}}{\sqrt{2}},
\]
and $\widetilde{K_e}$ is non-positive for $t_-\leq t\le t_+$ and negative for $t_- < t < t_+$. In this case, the curvature changes sign near the origin at $t = t_-$, and the boundary $\partial \mathcal C_e$ of the region $\mathcal C_e$ around $q_0$ where the electromagnetic curvature is positive is determined by the equation
\begin{equation}\label{level_set}
\alpha(\alpha - t_-^2)x^2 + \beta(\beta - t_-^2) y^2 = t_-^2\, e.
\end{equation}
The geometry of $\partial \mathcal C_e$ depends on the signs of the coefficients $\alpha - t_-^2$ and $\beta - t_-^2$. The transition occurs when
\[
\alpha = t_-^2
\quad \text{or} \quad
\beta = t_-^2,
\]
which is equivalent to the condition
\[
\sqrt{\alpha} + \sqrt{\beta} = \sqrt{2}.
\]
Moreover, as $e \to 0^+$, the sets $\partial C_e$ move closer and closer to the origin, and their topology is completely determined by the value of $\sqrt{\alpha} + \sqrt{\beta}$. More precisely, we have the following cases:
\begin{itemize}
    \item If $\sqrt{\alpha} + \sqrt{\beta} < \sqrt{2}$, the coefficients $\alpha - t_-^2$ and $\beta - t_-^2$ have opposite signs, so the quadratic form is indefinite. The set $\partial \mathcal C_e$ is given by the two branches of a \emph{hyperbola}, and as $e\to 0^+$ approaches two straight-line asymptotes. In the limit $e\to 0^+$, the region $\mathcal C_e$ remains wedge-shaped, meaning that there are always two cones of positive curvature (see Figure~\ref{fig:ricci_hyp}).

    \item If $\sqrt{\alpha} + \sqrt{\beta} > \sqrt{2}$, the coefficients $\alpha - t_-^2$ and $\beta - t_-^2$ have the same negative sign, and the quadratic form is negative definite. Thus, $\partial \mathcal C_e$ is an \emph{ellipse}, and, in the limit $e\to 0^+$, the region $\mathcal C_e$ collapses to the origin (see Figure~\ref{fig:ricci_ell}).

    \item In the borderline case $\sqrt{\alpha} + \sqrt{\beta} = \sqrt{2}$, one of the two coefficients vanishes, and the curve degenerates into a pair of \emph{parallel lines}. \qedhere
\end{itemize}
\end{proof}

\begin{figure}[h!]
\centering
\includegraphics[width=0.6\textwidth]{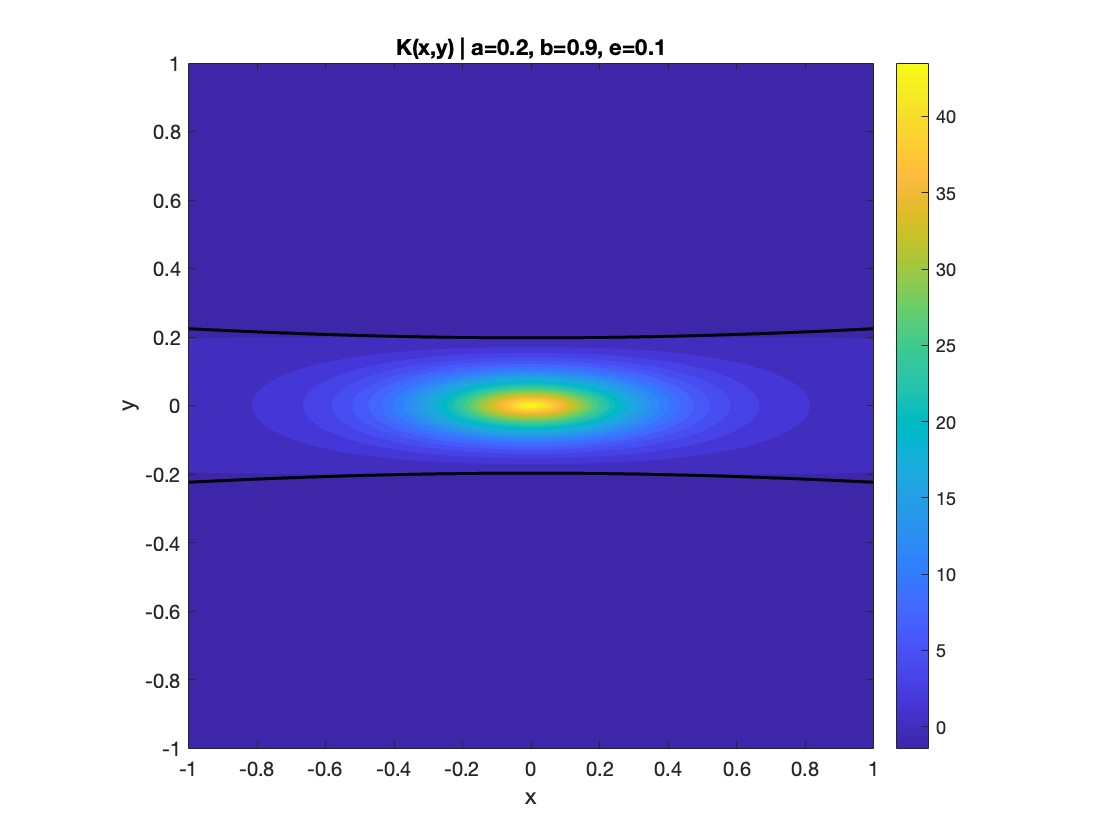}
\caption{If $\sqrt{\alpha} + \sqrt{\beta} < \sqrt{2}$, the positive curvature region is delimited by a hyperbola.}
\label{fig:ricci_hyp}
\end{figure}

\begin{figure}[h!]
\centering
\includegraphics[width=0.6\textwidth]{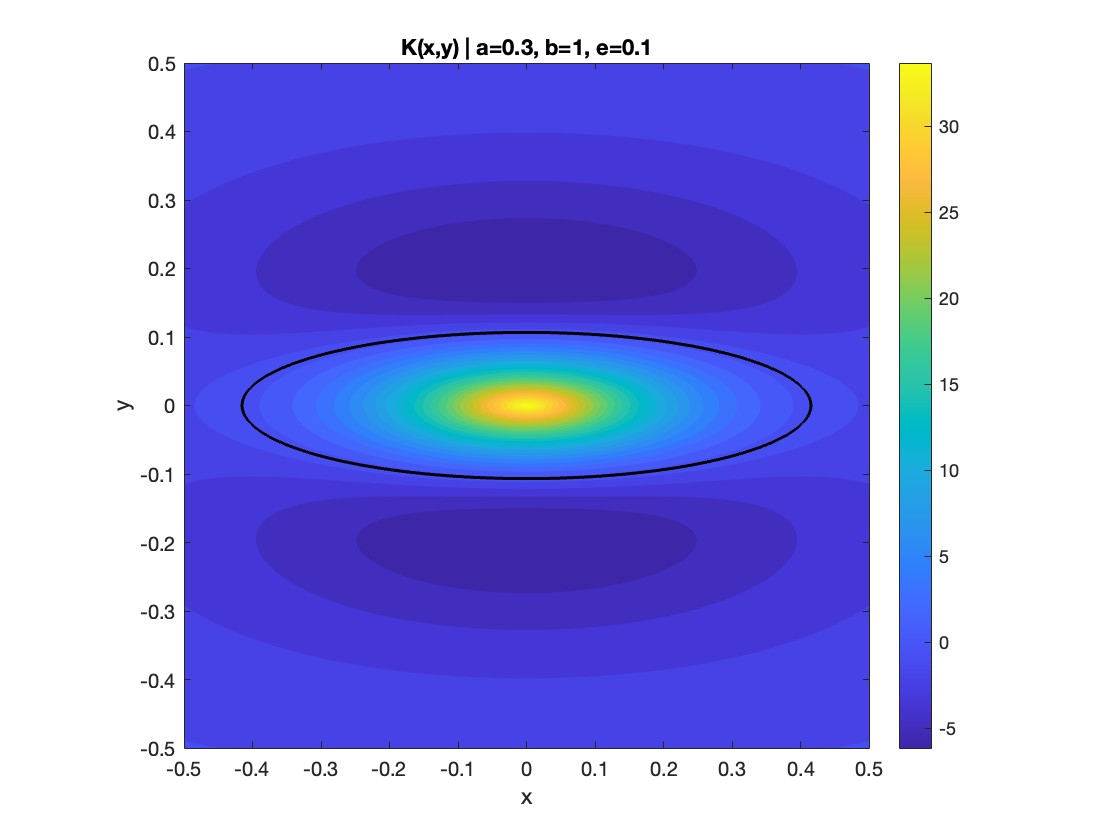} 
\caption{If $\sqrt{\alpha} + \sqrt{\beta} > \sqrt{2}$, the positive curvature region is delimited by an ellipse.}
\label{fig:ricci_ell}
\end{figure}

 \noindent Generalizing Theorem~\ref{prop:equivalenceplane} to linear electromagnetic systems in $\mathbb R^{2k}$, while important, is far from straightforward. In general, one cannot expect conclusions as sharp as in dimension two, unless the matrix defining the quadratic potential can be brought into $2\times 2$ block-diagonal form by a change of coordinates preserving the standard complex structure of $\mathbb R^{2k}$. In that case, the $2k$-dimensional problem decouples into $k$ two-dimensional ones, and Theorem~\ref{prop:equivalenceplane} applies to each block separately. We shall see in the next section that this block-diagonalization property, although restrictive, does occur for planar RE of the $n$-body problem in the presence of additional symmetries. Even when such a block-diagonalization is not available, however, one may still expect some aspects of Theorem~\ref{prop:equivalenceplane} to persist and to provide useful information on the stability properties of equilibrium points in $\mathbb R^{2k}$. While the equivalence between linear stability and the strict inequality $e_0<c_u$ is likely too much to expect (whereas the implication $e_0=c_u \Rightarrow$ linear instability is always true), the topology of the zero set of the electromagnetic curvature--possibly after restricting it to suitable $J$-invariant planes--is expected to encode relevant information on linear instability. Exploring these questions, already in dimension four, will be the subject of future work. A natural testing ground for these ideas is provided by Roberts' results on kite configurations \cite{roberts}, to which we return below.

\section{Applications to the planar $n$-body problem} \label{sec_5}

As shown in Section~\ref{sec_2.1}, RE of the $n$-body problem appear as equilibria of an electromagnetic Hamiltonian after passing to rotating coordinates. In the planar case, the angular velocity matrix generating the rotating frame is simply given by $K = k J$, with $k = \sqrt{\lambda}$. Denoting by $q_i = (x_i, y_i)$ the coordinates of each body, the magnetic potential is
\[
\vartheta_q = - M K q = k \sum_{i = 1}^n m_i (y_i\, dx_i - x_i\, dy_i),
\]
and the effective potential is
\[
V(q) = - U(q) - \tfrac12 k^2 \lvert q \rvert_M^2.
\]
In this final section, we analyze the stability of RE for the three- and four-body problems using electromagnetic curvature. By removing the symmetries of the problem and linearizing the potential at a RE, we obtain a reduced electromagnetic system with quadratic potential, to which the results of the previous sections can be applied.
The reduced Hamiltonian system is obtained using the symplectic change of coordinates discussed in Section~\ref{sec_3.2}, adapted here to the planar case (see also~\cite{meyer-schimdt}). Let $A \in \R^{2n \times 2n}$ be the $M$-orthogonal matrix commuting with $K$. The last two columns define a matrix
\[
B \in \R^{2n \times (2n - 4)},
\]
which identifies the reduced configuration space with the shape directions transverse to the symmetry directions. For $z \in \R^{2n-4}$, the configuration
\[
q = Bz
\]
represents the reduced shape coordinates in the original variables. Restricting the Hamiltonian to this subspace, the mass scalar product becomes the standard Euclidean scalar product in $\R^{2n-4}$, while the magnetic term becomes
\[
\vartheta_{\mathrm{red}}(z) = \vartheta(Bz) = k \sum_{i = 1}^{n-2}(y_i\,dx_i - x_i\,dy_i).
\]
The reduced potential is
\[
V_{\mathrm{red}}(z) = V(Bz) = -U(Bz) - \tfrac12 k^2 |z|^2.
\]
If $q_0$ is a RE, then, by construction of $B$, the corresponding shape coordinates are
\[
z_0 = B^\top M q_0 = 0.
\]
Linearizing the reduced potential at $z_0$ yields
\[
V_{\mathrm{red}}(z_0+\xi)
=
V_{\mathrm{red}}(z_0)
+
\tfrac12 \langle D^2 V_{\mathrm{red}}(z_0)\xi,\xi\rangle
+
\mathcal O(|\xi|^3),
\]
with
\[
D^2 V_{\mathrm{red}}(z_0)
=
B^\top D^2 V(q_0) B
=
- B^\top D^2 U(q_0) B - k^2 \mathbb{I}_{2n-4}.
\]
Thus, the linearized reduced dynamics is governed by an electromagnetic Hamiltonian on
$T^*\R^{2n-4} \cong \R^{4n-8}$ whose potential is quadratic. We are interested in those configurations for which the symplectic planes are invariant. In this case, the quadratic potential can be written in the form
\[
V_{\mathrm{lin}}(z) = \sum_{i = 1}^{n-2} (-\alpha_i x_i^2 - \beta_i y_i^2),
\]
where $x_i, y_i$ are coordinates spanning the symplectic planes, and $\alpha_i$ and $\beta_i$ are constants depending on the masses.

We start by analyzing RE for the 3-body problem associated with Lagrange equilateral triangle central configurations. In this case, after reduction by translations and rotations, the dynamics reduces from a Hamiltonian system on $T^* \R^6$ to one on $T^* \R^2$. Near a relative equilibrium $z_0$, the reduced Hamiltonian is determined by the quadratic approximation of the effective potential. Since the reduced configuration space has dimension two, it is always possible to choose coordinates $(x, y)$ in which the quadratic part of the potential takes the form
\[
V_\mathrm{lin}(z) = - \alpha x^2 - \beta y^2,
\]
with $\alpha$ and $\beta$ suitable constants. With this choice of coordinates, the linearized reduced dynamics is analogous to the toy model introduced in Section~\ref{sec_4.1}. The only difference lies in the magnetic term. In the toy model the magnetic potential is $\vartheta = -y\,dx + x\,dy$, while in the reduced three-body problem it becomes
\[
\vartheta_{\mathrm{red}} = k (y\,dx - x\,dy).
\]
Thus, the magnetic field strength is scaled by the factor $k$. In particular, the magnetic function becomes $b = 2k$, and the stability condition derived for the toy model rescales accordingly:
\[
\sqrt{\alpha} + \sqrt{\beta} < \sqrt{2}k.
\]

\begin{prop}
For the relative equilibrium corresponding to a Lagrange central configuration, the linear stability condition
\[
\sqrt{\alpha} + \sqrt{\beta} < \sqrt{2}\,k,
\]
is equivalent to Routh's classical criterion
\[
\frac{m_1 m_2 + m_1 m_3 + m_2 m_3}{(m_1 + m_2 + m_3)^2}
< \frac{1}{27}.
\]
Here $\alpha$ and $\beta$ are given by minus one half of the eigenvalues of the reduced Hessian of the effective potential.
\label{prop:lagrange}
\end{prop}

\begin{proof}
Consider three masses $m_1, m_2, m_3 > 0$ at the vertices of an equilateral triangle, with positions $q_1 = (1, 0)$, $q_2 = (-\tfrac12, -\tfrac{\sqrt{3}}{2})$, and $q_3 = (-\tfrac12, \tfrac{\sqrt{3}}{2})$, and denote by $q_0 = (q_1, q_2, q_3)$. Assume, without loss of generality, that $m_1+m_2+m_3 = 1$, and set
\[
\mu = m_1 m_2 + m_1 m_3 + m_2 m_3.
\]
After translating the Lagrange central configuration so that the center of mass is at the origin and normalizing it so that $q_0^\top M q_0 = 1$, we obtain
\[
k = \sqrt{U(q_0)} = (m_1 m_2 + m_1 m_3 + m_2 m_3)^{3/4} = \mu^{3/4}.
\]

To construct the matrix $B \in \R^{6\times 2}$ describing the reduced system, we first consider the vectors $v_1 = (1, 0, 1, 0, 1, 0)$, $v_2 = (0, 1, 0, 1, 0, 1)$, $v_3 = q$, and $v_4 = J q$. The vectors $v_1$ and $v_2$ correspond to translations, while $v_3$ and $v_4$ correspond to scaling and rotation of the configuration. We then choose a vector $v_5$ that is $M$-orthonormal to the four vectors $v_1, v_2, v_3, v_4$, and define $v_6 = J v_5$. For instance, one possible choice is
\[
v_5 = \frac{1}{\sqrt{\mu}} \Big( 0, \sqrt{\frac{m_2m_3}{m_1}}, \sqrt{\frac{3m_1m_3}{4m_2}}, -\sqrt{\frac{m_1m_3}{4m_2}}, -\sqrt{\frac{3m_1m_2}{4m_3}}, -\sqrt{\frac{m_1m_2}{4m_3}} \Bigr).
\]
Finally, we define
\[
B = \begin{bmatrix}
    v_5 & v_6
\end{bmatrix},
\]
whose columns span the reduced subspace corresponding to the shape dynamics.

Let $z_0$ denote the reduced coordinates of $q_0$:
\[
z_0 = B^\top M q_0 = (0, 0)\in \R^2.
\]
The Hessian of the reduced effective potential at $z_0$ is given by
\[
D^2 V_{\mathrm{red}}(z_0) = -B^\top D^2 U(q_0) B - k^2 \mathbb{I}_2
\]
and its eigenvalues are
\[
\lambda_{1,2} = - \tfrac32 k^2 \Bigl( 1 \pm \sqrt{1 - 3 \mu} \Bigr).
\]
Observe that, under the assumption $m_1 + m_2 + m_3 = 1$, it follows that $0 < \mu \leq \tfrac{1}{3}$, and therefore the quantity $\sqrt{1 - 3\mu}$ is real. Thus, both eigenvalues are real and negative. Therefore, we have
\[
\alpha = - \tfrac{\lambda_1}{2} = \tfrac34 k^2 \Bigl( 1 - \sqrt{1 - 3 \mu} \Bigr) > 0, \qquad \beta = - \tfrac{\lambda_2}{2} = \tfrac34 k^2 \Bigl( 1 + \sqrt{1 - 3 \mu} \Bigr),
\]
so that $\alpha, \beta > 0$. By squaring both sides of the linear stability condition
\[
\sqrt{\alpha} + \sqrt{\beta} < \sqrt{2}k
\]
we obtain
\[
\sqrt{\mu} < \tfrac{1}{3 \sqrt{3}},
\]
which is precisely Routh's stability criterion for the Lagrange triangular solution. Hence the two conditions are equivalent.
\end{proof}

Next, we analyze the linear stability of the RE of the four-body problem generated by the square and rhombus central configurations. The square is a central configuration for four equal masses, while the rhombus is central when opposite masses are equal. After reduction of the translational and rotational symmetries, the reduced configuration space has dimension $2n-4=4$, and the phase space is therefore
$T^*\R^4\cong\R^8$.
In these symmetric configurations, the reduced Hessian matrix splits into two $2\times2$ blocks corresponding to the invariant symplectic planes generated by the vectors $v_i, Jv_i$ of the matrix $B$. In coordinates adapted to these planes, the quadratic potential takes the form
\[
V_{\mathrm{lin}}(z)
=
-\alpha_1 x_1^2-\beta_1 y_1^2
-\alpha_2 x_2^2-\beta_2 y_2^2 .
\]
Therefore, the curvature criterion $\sqrt{\alpha_i} + \sqrt{\beta_i} < \sqrt{2}k$ can be applied separately to each plane. For both the square and the rhombus configurations, the stability condition given by Theorem~\ref{prop:equivalenceplane} is never satisfied on either plane, and the zero set of the electromagnetic curvature restricted to each plane is an ellipse. This agrees with the results obtained by Roberts for general kite configurations; see \cite{roberts}. Motivated by these examples, we formulate the following general principle.

\begin{prop}[Stability criterion]
Let $q_0$ be a relative equilibrium of the planar $n$-body problem for which the linearized reduced Hamiltonian splits into independent blocks on invariant symplectic planes. If the curvature stability condition given in Theorem~\ref{prop:equivalenceplane} fails on one of these planes, then the relative equilibrium is linearly unstable.
\label{prop:stabilitycriterion}
\end{prop}

The stability criterion established above, while conceptually appealing, is unfortunately hard to implement in general. Indeed, the assumption that the linearized reduced Hamiltonian splits into independent blocks on invariant symplectic planes is quite strong, and is typically satisfied only in the presence of additional symmetries. For example, among the kite configurations studied by Roberts in~\cite{roberts}, this phenomenon occurs only in the square and rhombus cases. In the general kite case, the reduced matrix still exhibits a sparse structure, but it cannot be brought into block-diagonal form by a change of coordinates preserving the complex structure.

Even in the absence of such a decomposition, however, the curvature point of view may still provide valuable information on linear stability. A natural strategy is to evaluate the electromagnetic curvature on $J$-invariant planes, even when these planes are not invariant under the linearized dynamics. This is expected to furnish at least necessary conditions for stability, and possibly sharper information in concrete examples. In particular, if the curvature test signals instability on one such plane, this should strongly suggest that the equilibrium itself is linearly unstable.
This perspective will be developed in a subsequent work. Among the cases we intend to study are: central configurations with $n \geq 5$ and a large symmetry group, where one expects the reduced system to split and the two-dimensional curvature criterion to apply directly; the kite configurations considered by Roberts, where one hopes at least to recover instability outside the ``dominant mass'' regime by means of the curvature test; and relative equilibria in the planar four-body problem arising from configurations more general than the kite ones.
One notable advantage of this approach is that it is based on a coordinate change that is far less \emph{ad hoc} than the one introduced by Roberts, since it ultimately follows the reduction procedure of Meyer and Schmidt~\cite{meyer-schimdt}. This makes the method potentially applicable in a much wider range of situations. A second advantage is numerical: instead of requiring the computation of eigenvalues together with a delicate analysis of diagonalizability, one only needs to evaluate the electromagnetic curvature via the associated two-dimensional tests.

\bibliographystyle{plainurl}
\bibliography{biblio} 

\end{document}